\documentclass[12pt,twoside]{article}
\usepackage{amsmath,amssymb}    
\usepackage{url}                
\usepackage{graphicx}           
\usepackage{amsthm}
\usepackage{verbatim}
\theoremstyle{definition}
\newtheorem{theorem}{Theorem}[section]
\newtheorem{lemma}[theorem]{Lemma}
\newtheorem{proposition}[theorem]{Proposition}
\newtheorem{corollary}[theorem]{Corollary}

\newtheorem{definition}[theorem]{Definition}
\newtheorem{conjecture}[theorem]{Conjecture}

\newtheorem{note}[theorem]{Note}

\newcommand{\la}{\lambda}
\newcommand{\sig}{\sigma}
\newcommand{\ul}{\underline}
\newcommand{\germ}{\mathfrak}
\newcommand{\RR}{\mathbb{R}}

\newcommand{\ZZ}{\mathbb{Z}}
\newcommand{\cH}{\mathcal{H}}
\newcommand{\sss}{\mathfrak{s}}

\newcommand{\ttt}{\mathfrak{t}}

\newcommand{\hh}{\mathfrak{h}}
\newcommand{\OO}{\mathcal{O}}

\date{}
\begin{document}

\pagestyle{myheadings}
\title{Unitary representations of Hecke algebras of complex reflection groups}
\author{Emanuel Stoica}
\maketitle

\begin{abstract}
We begin the study of unitary representations of Hecke algebras of complex reflections groups. We obtain a complete classification for the
Hecke algebra of the symmetric group $\mathfrak{S}_n$ over the complex numbers. Interestingly, the unitary representations in the category $\OO$ for the rational
Cherednik algebra of type A studied in \cite{ESG} correspond to the unitary  representations of the corresponding Hecke algebra via
the KZ functor defined in \cite{GGOR}.

\end{abstract}

\section{Introduction}

	In this paper, we begin the study of unitary representations of Hecke algebras of complex reflection groups. Given a complex reflection group $W$ and the corresponding Hecke algebra $\cH=\mathcal{H}_{\ul{q}}(W)$ over
the complex numbers, we investigate the existence of a certain $\cH$-invariant nondegenerate  Hermitian form on a representation $V$. The defining property of the form is its invariance under the braid group $B_W$ of $W$, namely
$(Tv,v')=(v,T^{-1}v')$ for all $T \in B_W$ and $v,v' \in V$.

A representation with such a Hermitian form will be called \emph{unitary} if the form is positive definite. Since unitary representations are semisimple, we may restrict the study of unitarity to irreducible representations. 

The main result is Theorem \ref{Main} which provides a complete classification of the unitary irreducible representations of the Hecke algebra of the symmetric group $\mathfrak{S}_n$. The proof uses the theory of tensor categories and properties of restriction functors.

The organization of the paper is the following. In section 2, we overview
a few  general definitions and results regarding the Hecke algebra $\cH$ of  the complex reflection group $G(r,1,n)$; in section 3, we construct a Hermitian form on Specht $\cH$-modules and their irreducible quotients invariant under the action of the corresponding braid group; finally, in section 4, we classify the unitary irreducible representations of the Hecke algebra of the symmetric group $\mathfrak{S}_n$.

{\bf Acknowledgments.}
The author is grateful to his advisor P. Etingof for his continuous and invaluable assistance. This research was supported in part by the 
NSF grant DMS-0504847.

\section{Preliminaries}
Let $r$ and $b$ be positive integers. We will first review the definition of the Hecke algebra of the complex reflection group $G(r,1,n)$, namely the wreath product $(\ZZ/r\ZZ)^n \rtimes \mathfrak{S}_n$. This algebra is also known as the \emph{Ariki-Koike algebra}.

\begin{definition} Given a commutative domain $R$ and a family of parameters $\ul{q}=(q,q_1,\dots, q_r) \in R^{r+1}$,
the Hecke algebra $\cH_{\ul{q}}(W)$ of the complex reflection group $W=(\ZZ/r\ZZ)^n \rtimes \mathfrak{S}_n$ is the algebra generated by $T_0,T_1,\dots,T_{n-1}$ and relations
\begin{eqnarray*}
(T_0-q_1)\ldots (T_0-q_r)=0, \\
T_0T_1T_0T_1=T_1T_0T_1T_0, \\
(T_i+1)(T_i-q)=0 &\text{for all}& i=1,\dots,n-1, \\
T_iT_{i+1}T_i=T_{i+1}T_iT_{i+1} &\text{for all}& 1 \leq i \leq n-2,\\
T_iT_j=T_jT_i &\text{for all}& 0\leq i\leq j-2 \leq n-3.
\end{eqnarray*}
\end{definition}

Note that $T_1,\dots, T_{n-1}$ generate a subalgebra of $\cH_{\ul{q}}(W)$ isomorphic to the Hecke algebra
$\cH_q(\mathfrak{S}_n)$ of the symmetric group. It has a natural basis $\{T_w|\text{ } w \in \mathfrak{S}_n\}$ over $R$ indexed by permutations, such that every $T_w$ is given by certain a product of elements $T_i$ for $1 \leq i \leq n-1$.

In our paper, we will assume $R=\mathbb{C}$. For 'most' complex values of the parameters $\ul{q}=(q,q_1,\dots, q_r)$, the Hecke algebra $\cH_{\ul{q}}(W)$ is semisimple. We will call these values \emph{generic}. 

We assume all parameters are invertible. We will use extensively the \emph{involution} $\sigma$, a semilinear involution of $\cH= \cH_{\ul{q}}$ that sends $a \mapsto \overline{a}$, $q \mapsto q^{-1}$, $q_i \mapsto q_i^{-1}$,  $T_0 \mapsto T_0^{-1}$, and $T_i \mapsto T_i^{-1}$,  for all $a \in \mathbb{C}$ and $1 \leq i \leq n-1$. This defines an involution for the universal Hecke algebra, where $\underline{q}$ are taken to be (invertible) formal parameters. When we specialize the parameters to complex numbers on the unit circle, this map can be extended to a $\mathbb{C}$-semilinear involution of the special Hecke algebra. Note that we must assume the parameters are on the unit circle to be able to define this involution for specializations to complex numbers.

We also define the $\mathbb{C}$-linear \emph{anti-involution} $^*$ on $\cH$ that sends $q \mapsto q$, $q_i \mapsto q_i$,  $T_0 \mapsto T_0$, and $T_i \mapsto T_i$, for $1 \leq i \leq n-1$. This determines an anti-involution for any specialization of variables to complex numbers.

We will follow \cite{DJM} in defining for any multipartition $\la=(\la^1,\dots,\la^r)$ of $n$ a Specht module $S^{\la}$. Given two multipartitions $\la$ and $\mu$, we say that $\la$ \emph{dominates} $\mu$ if $\sum_{i=1}^{k-1} |\la^i|+\sum_{i=1}^j \la^k_i \geq \sum_{i=1}^{k-1}|\mu^i|+\sum_{i=1}^j \mu^k_i$ for all $1\leq k \leq r$ and $j\geq 1$. If $\la$ dominates $\mu$ and $\la \neq \mu$ we write $\la \rhd \mu$. 

Let $\textbf{t}^{\la}$ be the standard $\la$-tableau 
filled with $1,2, \dots, n$ in order in the first row, second, and so on,
in $\la^1,\dots ,\la^r$. For any standard $\la$-tableau $\sss$, define $d(\sss) \in \mathfrak{S}_n$ to be the permutation such that $\mathbf{t}^{\la} d(\sss)=\sss$. It is straightforward that the set $\{d(\sss)| \sss \text{ standard } \la\text{-tableau}\}$ is in one-to-one correspondence with (right) coset representatives of $\mathfrak{S}_{\la}$ in $\mathfrak{S}_n$ where $\mathfrak{S}_{\la}=\mathfrak{S}_{|\la^1|} \times \mathfrak{S}_{|\la^2|} \times \ldots \mathfrak{S}_{|\la^r|}$ is a Young subgroup of $\mathfrak{S}_n$.  

Given the multipartition $\la$ of $n$, let $\ul{a}=(a_1,\dots,a_r)$ 
where $a_i=\sum_{j=1}^{i-1} |\la^{j}|$ for any $1 \leq i \leq r$.
Let also $x_{\la}=\sum_{w \in \mathfrak{S}_{\la}} T_w$ and
$$L_m=q^{1-m}T_{m-1}\dots T_1 T_0 T_1 \dots T_{m-1}$$ for $m=1,\dots ,n$.
We define $m_{\la}=(\prod_{k=1}^r \prod_{m=1}^{a_k} (L_m-q_k)) x_{\la} \in \cH$. For any pair of standard 
$\la$-tableaux $(\sss ,\ttt)$, set $m_{\sss \ttt}=T_{d(\sss)}^* m_{\la}T_{d(\ttt)}$.

Now, let us define the vector space $$
\bar{N}^{\la}=\mathbb{C}\{m_{\mathfrak{s} \mathfrak{t}} | \sss, \ttt \text{ standard } \mu\text{-tableaux } \text{for a multipartition }\mu \text{ of }
n, \text{ } \mu \rhd \la\} $$ which is a two-sided ideal in $\cH$ (see \cite{DJM} Proposition 3.22).

\begin{definition}
Define $z_{\la}=(\bar{N}^{\la}+m_{\la})/\bar{N}^{\la}$ to be an element in $\cH/\bar{N}^{\la}$ and the Specht module $S^{\la}=\cH z_{\la}$. 
\end{definition}

\begin{note}
In \cite{DJM}, the Specht module $S^{\la}$ is defined as a right module $z_{\la}\cH$, but we  prefer the opposite definition. The anti-involution $^*$ maps one to the other, since $m^*_{\la}=m_{\la}$ and $(\bar{N^{\la}})^*=\bar{N^{\la}}$.
\end{note}
It is well known that for generic values of the parameters, the Specht modules corresponding to multipartitions of $n$ are irreducible and exhaust all irreducible representations of $\cH$.

\section{The existence of the Hermitian form}

In this section, we assume $\cH=\cH_{\ul{q}}(W)$ is the Hecke algebra of the complex reflection group $G(r,1,n)$ defined  above. We prove that if parameters are complex numbers on the unit circle, any irreducible representation $V$ of $\cH$ is
Hermitian, namely, it admits a nondegenerate Hermitian form such that $(Tv,v')=(v,\sigma(T^*) v')$ for all $v,v' \in V$ and $T \in \cH$. Equivalently, the form is invariant under the corresponding braid group.

In order to construct the Hermitian form, we will use the symmetric bilinear form $\langle,\rangle$ on the Specht module $S^{\la}$ defined in \cite{DJM}, which satisfies 
$\langle Tv,v'\rangle=\langle v,T^*v'\rangle$ for every $v,v' \in S^{\la}$, and $T \in \cH$. If the Hecke algebra is semisimple, the Specht module is irreducible and the form is nondenegerate. If the Hecke algebra is not semisimple, the form may degenerate, and the quotient $D^{\la}=S^{\la}/rad\langle,\rangle$ will be either irreducible or zero. The family  $\{D^{\la} | \text{ } \la \text{ multipartition of } n, D^{\la} \neq 0\}$ forms a complete set of irreducible representations of $\cH$ (for more details, see \cite{DJM}).

For any multipartition $\la$ of $n$, the following holds

\begin{lemma} \label{lemma} (i) The involution $\sigma$ preserves 
$\bar{N}^{\la}$, and $\sig(m_{\la})=m_{\la}u$ for some invertible element $u \in \cH$;

(ii) The involution $\sigma$ preserves the Specht modules, namely $\sig(S^{\la})=S^{\la}$ for any $\la$.
\end{lemma}
\begin{proof}

Let us first prove that $\sig(x_{\la})=q^s x_{\la}$ for some integer $s$. It is easy to observe from the definition of $x_{\la}$ that we can reduce the problem to the case $x=\sum_{w \in \mathfrak{S}_n} T_w$. Note that $T_w x=q^{l(w)}x$ for all $w \in \mathfrak{S}_n$ where $l(w)$ is the length of $w$. It is well known that $x$ is unique with such property, up to scaling (see, for example, \cite{DJ1}, section 3). Applying $\sig$, we obtain $T_w \sig(x)=q^{l(w)}\sig(x)$ which, combined with the uniqueness of $x$, gives $\sig(x)=Px$ where $P \in \mathbb{C}[q,q^{-1}]$.
Since $\sig$ is an involution, $P(q)P(q^{-1})=1$ which implies $P=q^{s}$ for some integer $s$. 

Since $m_{\la}=\prod_{k=1}^r (\prod_{m=1}^{a_k} (L_m-q_k)) x_{\la} \in \cH$, a straightforward computation taking into account the commutativity of $L_m$ for $1\leq m \leq n$ gives $\sig(m_{\la})=m_{\la}u$ where $u=q^t \prod_{k=1}^r L_k^{b_k} q_k^{-a_k}$ for some integers $t$, $a_k$, $b_k$, and $1\leq k \leq r$. Note that $u \in \cH$ is invertible. Finally, we note that $\bar{N}^{\la}=\sum_{\nu \rhd \la} \cH m_{\nu} \cH$, since it is a two-sided ideal in $\cH$. Hence, $\sig$ preserves $\bar{N}^{\la}$, and also $S^{\la}$.
\end{proof}

\begin{proposition} \label{form} If the complex parameters $\ul{q}$ are on the  unit circle, then any irreducible representation $D^{\la} \neq 0$ has a nondegenerate 
Hermitian form such that 
\begin{equation} \label{invar}
(T v, v')=(v,\sigma(T^*)v')
\end{equation}
for all $v, v' \in D^{\la}$ and $T \in \cH$.
\end{proposition}
\begin{note} For any $0 \leq i \leq n-1$, $\sigma(T_i^*)=T_i^{-1}$. Moreover, $\sigma(T^*)=T^{-1}$ for any $T$ from the group generated by elements $T_i$.
\end{note}

\begin{proof}
First, let us recall the existence of a symmetric bilinear form $\langle,\rangle$ on $S^{\la}$  such that $\langle T v, v' \rangle=\langle v,T^* v' \rangle$ for any $v,v' \in D^{\la}$ and $T \in \cH$. This form descends to a nondegenerate form on $D^{\la}$.

Let us note that there is a semilinear involution on the Specht module $S^{\la}$, which by abuse of notation we will also denote by $\sig$, such that  $\sig(T) \sig(v)=\sig(Tv)$ for any $T \in \cH$ and $v \in S^{\la}$. On $S^{\la}$, we construct a new form $\langle v,w\rangle_1=\langle v,\sig(w)\rangle$. Clearly, the form $\langle, \rangle_1$ is sesquilinear and descends to a nondenegerate form on $D^{\la}$. Since the form $\langle, \rangle_1$ is unique up to rescaling, it follows that $\langle v,v' \rangle_1=c \overline{\langle v', v \rangle}_1$ for all $v, v' \in D^{\la}$ and a complex constant $c$. Note that $c$ must be on the unit circle, hence $c=exp(i\theta)$ for some $\theta \in \RR$. Define now $(v, v')=\alpha \langle v, v' \rangle_1$ where $\alpha=exp(i \frac{\theta}{2})$. We conclude immediately that $(,)$ is Hermitian.
 \end{proof}

We end this section with the following
\begin{definition}
Given a multipartition $\la$ such that $D^{\la} \neq 0$, we say the   $D^{\la}$ is \emph{unitary} if the nondegenerate Hermitian form defined above  can be normalized to be positive definite. 
\end{definition}

\section{Unitary representations in type A}

The rest of the paper will focus on the Hecke algebra of the symmetric group $\mathfrak{S}_n$. In this case, the Hecke algebra depends only on
one parameter $q$. For a complex $q$, let $e$ be the smallest positive integer such that $1+q+\ldots +q^{e-1}=0$. If such an integer does not exist, we set $e=\infty$.

The irreducible representations $D^{\la}$ of $\cH_{q}$ are indexed by partitions of $n$. As it is well known, $D^{\la} \neq 0$ if and only if $\la^{\vee}$ is $e$-restricted where $\la^{\vee}$ is the conjugate partition. In other words, $\la^{\vee}_i - \la^{\vee}_{i+1}<e$ for all $i$ . In this section, we will denote $D_{\la} \colon=D^{\la^{\vee}}$.
Hence, $D_{\la} \neq 0$ if and only if $\la$ is $e$-restricted. Similarly, we will denote $S_{\la}=S^{\la^{\vee}}$. As mentioned in the previous section, $D_{\la}=S_{\la}/rad(,)$.

We define the unitary locus of $\la$ to be the set $$U(\la)=\{c \in (-\frac{1}{2},\frac{1}{2}] | D_{\la} \text{ is nonzero and unitary at } q=exp(2 \pi i c)  \}. $$

\begin{proposition} \label{hook} (i) For any $n \geq 2$,  $U(1^n)=(-\frac{1}{2},\frac{1}{2}]$ and $$U(n)=(-\frac{1}{2},\frac{1}{2}] \setminus \{\pm\frac{r}{m} | 1 \leq r,m \leq n, gcd(r,m)=1\}.$$

(ii) If $n \geq 3$ and $\la=(n-k,1^k)$ for $1\leq k \leq n-2$, then  $U(\la)=[-\frac{1}{n},\frac{1}{n}].$
\end{proposition}

\begin{proof} 
Given $c \in (-\frac{1}{2},\frac{1}{2}]$, we set $q=exp(2\pi i c)$ and $e$ the smallest positive integer such that $1+q+\ldots+q^{e-1}=0$.

If $\la=(n)$ is $e$-restricted, then $D_{(n)} \neq 0$ is one-dimensional and, therefore, unitary. Similarly, $D_{(1^n)}$ is one-dimensional and unitary.

Let us, now, assume that $n \geq 3$ and $\la=(n-k,1^k)$ for some $1\leq k \leq n-2$. The form on $S_{\la}$ is nondegenerate, if $c \in (-\frac{1}{n},\frac{1}{n})$, and positive definite at $c=0$, hence $(-\frac{1}{n},\frac{1}{n}) \subset U(\la)$. At $c=\pm\frac{1}{n}$, the form is positive definite on the quotient $D_{\la}=S_{\la}/rad(,)$, therefore $[-\frac{1}{n},\frac{1}{n}] \subset U(\la)$.

We will use induction on $n$ and the restriction functor $Res$ given by the natural inclusion $\cH(\mathfrak{S}_{n-1})\subset \cH(\mathfrak{S}_n)$ to prove that $U(\la) \subset [-\frac{1}{n},\frac{1}{n}]$. 

First, we will use the induction hypothesis to show that $U(\la) \subset [-\frac{1}{n-1},\frac{1}{n-1}]$. For $n=3$, this is automatic, so we may assume $n>3$. Let $c \in U(\la)$ and assume $\la$ is $e-$restricted. At least one of the partitions $(n-k-1,1^k)$ and $(n-k,1^{k-1})$ is different from $(n-1)$ and $(1^{n-1})$. Let us denote one such partition by $\nu$. We observe that $D_{\nu} \neq 0$, since $\nu$ is $e$-restricted. Using Theorem 2.5 in \cite{Br}, the module $D_{\nu}$ is contained in the composition series of $Res D_{\la}$. Since $D_{\la}$ is unitary, so is $D_{\nu}$. By the induction hypothesis, $c \in [-\frac{1}{n-1},\frac{1}{n-1}]$.

We will now show that $D_{\la}$ is not unitary for  $ c\in [-\frac{1}{n-1},-\frac{1}{n}) \cup (\frac{1}{n},\frac{1}{n-1}]$.
Using an argument similar to that in \ref{form}, we can define a Hermitian form
$(,)$ on the universal Specht module $\tilde{S}_{\la}$ for the universal Hecke algebra. It will be related to the symmetric form $\langle, \rangle$ defined in \cite{DJ1,DJ2} by  $(v,v')=q^{-k/2}\langle v,\sigma(v') \rangle$ for any $v,v' \in \tilde{S}_{\la}$ and  a fixed integer $k$ depending on $\la$. Specializing at generic complex values of $q$, we obtain the Hermitian form on $S_{\la}$, up to rescaling. Since for $c \in (-\frac{1}{n},\frac{1}{n})$ the module $S_{\la}$ is irreducible and its form is nondegenerate, we can normalize the form on $\tilde{S}_{\la}$, such that it is positive definite upon specialization in this interval. 

Note that $S_{\la}$ is irreducible for any $ c\in \pm(\frac{1}{n},\frac{1}{n-1}]$ because the $e$-core of $\la$ is itself (see \cite{DJ2}, Theorem 4.13 for details). Hence, the signature of the form on $S_{\la}$ does not change when $ c \in \pm(\frac{1}{n},\frac{1}{n-1}]$, because the form on $\tilde{S}_{\la}$ does not degenerate when specialized in this interval. It suffices to prove that this form is not positive definite upon specialization in a small neighborhood $\pm(\frac{1}{n},\frac{1}{n}+\epsilon)$. Let us look at the signature of the form in this interval. Let $\zeta=e^{2\pi i/n}$ and $M_i=\{v \in \tilde{S}_{\la} | (q-\zeta)^i \text{ divides } (v,-)\}$. The Jantzen filtration of the Hermitian form given by $\tilde{S}_{\la}=M_0 \supset M_1 \supset \ldots $ is the same as for the symmetric form $\langle, \rangle$.  Looking at the determinant of the latter computed in \cite{DJ2} Theorem 4.11, we observe that it has a root $\zeta$ of multiplicity $n-2 \choose k$. This equals the dimension of $D_{\mu}$ at $c=\pm \frac{1}{n}$ where $\mu=(n-k-1,1^{k+1})$. At $c=\pm \frac{1}{n}$, it is known that $S_{\la}$ has length 2, its socle is $D_{\mu}$, and its head $D_{\la}$. Hence, specializing at $c=\pm \frac{1}{n}$, the Jantzen filtration becomes $0=M_2 \subsetneq M_1=D_{\mu} \subsetneq S_{\la}$. Therefore, the signature of the form changes when crossing 
$c=\pm \frac{1}{n}$, without becoming negative definite. We conclude that the form ceases to be positive definite in the interval $\pm(\frac{1}{n},\frac{1}{n}+\epsilon)$, which finishes the proof.
\end{proof}

We will now state and prove the main theorem.
Let $\la$ be a partition, $L$ its largest hook, and $b$ the multiplicity
of the largest part. If $\la$ is not rectangular with $b>1$, let $l=L-b+1$, otherwise, let $l=L-b+2$.  For any $\la$, we will call the integers $l \leq k \leq L$ the \emph{main hooks} of $\la$. Note that if $k$ is a main hook of $\la \neq (n)$, then $\la$ is $k$-restricted.

\begin{theorem} \label{Main}
If $\la \neq (n), (1^n)$, the unitary locus is given by
$U(\la)=(-\frac{1}{L},\frac{1}{L}) \cup \{\pm \frac{1}{k} |\text{ } k  \text{ is a main hook of } \la\}$.

\end{theorem}
\begin{proof}
The proof will be contained in the following two propositions.
\end{proof}

\begin{proposition}
If $\la \neq (n), (1^n)$, then $U(\la)$ is contained in the set $(-\frac{1}{L},\frac{1}{L}) \cup \{\pm \frac{1}{k} |  \text{ } 
k \text{ is a main hook}\}$.

\end{proposition}
\begin{proof}
Let us assume $n \geq 4$ and $\la \neq (n-k,1^k)$ for all $0 \leq k \leq n-1$, otherwise the result follows from \ref{hook}. We will use the restriction functor coming from the natural inclusion  $\cH(\mathfrak{S}_{n-1}) \subset \cH(\mathfrak{S}_n)$. Let $c \in U(\la)$, $q=exp(2\pi i c)$ and $e$ the smallest positive integer such that $1+q+\ldots+q^{e-1}=0$. 

First, assume that $e>L$. We will prove that $c \in (-\frac{1}{L},\frac{1}{L})$. Using the restriction functor, we obtain 
$Res S_{\la}=\oplus_{\nu \rightarrow \la} S_{\nu}$ where the arrow indicates the addition of a box. Let $\nu$ be a partition in this sum obtained by removing a box from $\la$ situated away from the largest hook. Since $S_{\la}$ is unitary, so is $S_{\nu}$. The largest hook of $\nu$ is $L$, hence $c \in (-\frac{1}{L},\frac{1}{L})$ by the induction step.

Let us now assume that $2 \leq e \leq L$, $c=\frac{r}{e}$ for some integer $r$ relatively prime to $e$. Note that $\la$ is $e$-restricted. 
From Theorem 2.6 in \cite{Br}, the socle of the restriction of $D_{\la}$ to $\cH(\mathfrak{S}_{n-1})$ is
$$Soc(Res D_{\la})=\oplus_{\substack{\mu \xrightarrow{good} \la}} \mu$$ where the arrow indicates the addition of a \emph{good} box (see \cite{Br} for details).
For all partitions $\mu$ in the sum, $D_{\mu} \neq 0$ and unitary. 

We distinguish the following cases. First, we assume there exists a good corner
not situated on the largest hook of $\la$. Let $\nu$ be the partition obtained by removing this corner. Applying the induction hypothesis to $D_{\nu}$, we conclude that $l \leq e \leq L$ and $r=\pm 1$.

Assume, now, that all good corners are on the largest hook. 
If there is a good corner on the first column, we remove it and obtain a partition $\nu$. Applying the induction hypothesis, we obtain $r=\pm 1$ and $l-1 \leq e \leq L$. Assuming $e=l-1$, we note that the highest corner of $\la$ is good, hence it must be on the largest hook. In addition, $\la$ has no corners outside the largest hook, because they would be good. Hence, $\la$ is of the form $(n-k,1^k)$, a contradiction. Therefore, $l \leq e \leq L$.

Assume, now, there is a good corner only on the first column. By removing this corner, we obtain a partition $\nu$ such that $D_{\nu}$ is unitary. Hence $r=\pm 1$ and $a \leq e \leq L$ where $a$ is the smallest main hook of $\nu$. Assume, first, that the highest corner of $\nu$ is also a corner for $\la$. This must be a \emph{normal} corner of $\la$ (see \cite{Br} for the definition). It is not situated on the largest hook of $\la$, otherwise, $\la=(2,1^{n-1})$. Let us remove this corner and obtain a partition $\nu'$. Note that $D_{\nu'} \neq 0$. Using Theorem 2.5 in \cite{Br}, $D_{\nu'}$ belongs to the composition series of $Res D_{\la}$, hence it is unitary. Applying the induction hypothesis, we obtain the desired result.

Finally, if the highest corner of $\nu$ is not a corner of $\la$, then $\la_1-\la_2 \geq 2$ and $L-1 \leq e \leq L$. Then there is no corner outside of the largest hook, because it would be good. Hence,  $\la$ is of the form $(n-k,1^k)$, a contradiction.  This finishes the proof.
\end{proof}
\begin{proposition}
If $\la \neq (n), (1^n)$, then $(-\frac{1}{L},\frac{1}{L}) \cup \{\pm \frac{1}{k} | \text{ } 
k \text{ is a main hook}\}$ is contained in the unitary locus $U(\la)$.
\end{proposition}

\begin{proof}
First note that $(-\frac{1}{L},\frac{1}{L}) \in U(\la)$, since the form on $S_{\la}$ is nondegenerate on this interval and it is positive definite at $c=0$. We will now prove that $D_{\la}$ is unitary at $\pm \frac{1}{l+k}$ for all nonnegative integers $k$ where $l$ is the smallest main hook. 

Let $N \geq 2$ and $a$ be positive integers
and $Q$ a primitive root of unity of order $2(a+N)$.
Then one can define the fusion category $\mathcal{C}_Q$ of representations of the quantum group
$U_Q(sl_N)$. It is a modular tensor category whose simple objects
are highest weight representations with dominant integral highest weight 
$\mu$ such that $(\mu,\theta)\leq a$ where $\theta=(1,0,...,0,-1)$ is the highest root. For details, we refer to \cite{BK} and the references therein.

The category $\mathcal{C}_Q$ is unitary for $Q=exp(\pm \pi i/(a+N))$ and Hermitian for any $Q$. For this fact, and the general notions of Hermitian and unitary categories, see \cite{K1,K2} and the references therein. Since $\mathcal{C}_Q$ is unitary for such $Q$, the braid group representations on 
multiplicity spaces of tensor powers of an object are also unitary. In 
particular, if $V$ is the vector representation, this is true for the power $V^{\otimes n}$.

Now, the braid group representations  on  multiplicity spaces of this 
tensor power  factor through the Hecke algebra $\cH_{Q^2}$.
From the quantum Schur-Weyl duality we obtain $$V^{\otimes n}=\oplus_{\mu: (\mu,\theta) \leq a } \text{ }\pi_\mu\otimes D_{\mu^\vee}$$
 where $\pi_\mu$ are simple objects of $\mathcal{C}_Q$ 
and $D_{\mu^{\vee}}$ are irreducible representations of the Hecke algebra $\cH_{Q^2}$. 

If we think of $\mu$ as a partition, then it is a partition of $n$ with
at most $N$ non-zero parts such that $\mu_1-\mu_N \leq a$. Equivalently, $\lambda=\mu^\vee$ is a partition of $n$ with largest part at most $N$, smallest main hook $l$, and $l-N \leq a$.  

Returning to our problem, given a partition $\la \neq (n),(1^n)$ with smallest main hook $l$, let $N \geq 2$ be its largest part and $a=l-N+k$ for a nonnegative integer $k$. Note that $a>0$. It follows from above that $D_{\la}$ is unitary at $q=exp(\pm 2\pi i/(l+k))$. This finishes the proof.
\end{proof}

The KZ functor defined in \cite{GGOR} maps representations  from the category $\OO$ of the the rational Cherednik algebra $H_c(W)$ to representations of the Hecke algebra
$\cH_q(W)$ where $q=e^{2 \pi i c}$. This functor maps the irreducible lowest weight module $L_c(\la)$ to the irreducible representation $D_{\la}$, if $c>0$, and to $D_{\la^{\vee}}$, if $c<0$. Note that $L_c(\la)$ is mapped to $0$ if $\la$ is not $e$-restricted.
Using the description of the unitary irreducible representations in the category $\OO$ presented in
\cite{ESG}, we obtain the following

\begin{corollary}
The KZ functor maps unitary representations from the category $\OO$ for the rational Cherednik algebra of type A to unitary representations of the Hecke algebra of type A or to zero. Moreover, all unitary representations of the Hecke algebra are obtained in this way.
\end{corollary}

\begin{note}
The only irreducible unitary representations mapped to zero by the KZ functor are those corresponding to a rectangular partition at $c=\frac{1}{l}$ where $l$ is the width of the rectangle.
\end{note}

We conjecture that the corollary can be generalized to the restriction functor defined in \cite{BE} for all complex reflection groups. 
Let $W$ be a complex reflection group, $\hh$ its reflection representation, and $W'$ a parabolic subgroup of $W$, namely the stabilizer of a point in $\hh$. Let also $H_c(W,\hh)$ be the corresponding rational Cherednik algebra, $c'$ the restriction of the parameter $c$ to $W'$, and $Loc(\hh_{reg}^{W'})$ the category of local systems on $\hh_{reg}^{W'}$. There is a restriction functor 
$$Res \colon \OO_c(W,\hh) \rightarrow \OO_{c'}(W',\hh/\hh^{W'}) \boxtimes Loc(\hh_{reg}^{W'})$$  

Composing it with the functor $Mon \colon Loc(\hh_{reg}^{W'}) \rightarrow Rep  \text{ }\pi_1(\hh_{reg}^{W'})$ which sends local systems on $\hh_{reg}^{W'}$ to representations of the fundamental group $\pi_1(\hh_{reg}^{W'})$, we obtain the functor
$$Res' \colon \OO_c(W,\hh) \rightarrow \OO_{c'}(W',\hh/\hh^{W'}) \boxtimes Rep   \text{ }\pi_1(\hh_{reg}^{W'}).$$

Let $I$ be a finite set. If $M_i \in \OO_{c'}(W',\hh/\hh^{W'})$ and $V_i \in Rep \text{ }\pi_1(\hh_{reg}^{W'})$ for all $i \in I$, we say that the object $\oplus_{i\in I} M_i \boxtimes V_i  \in \OO_{c'}(W',\hh/\hh^{W'}) \boxtimes Rep \text{ }\pi_1(\hh_{reg}^{W'})$ is \emph{unitary} if and only if $M_i$ and $V_i$ are unitary in their categories.

\begin{conjecture}
If $M \in \OO_c(W, \hh)$ is unitary then so is $Res'(M)$.
\end{conjecture}

In this paper, we have proved the conjecture for $W=S_n$, $\germ{h}$ its reflection representation and $W'=1$.

\end{document}